\documentclass{amsart}
\usepackage{graphicx} 
\usepackage{amsmath}
\usepackage{amssymb}
\usepackage{amsthm}
\usepackage{float}
\usepackage[backend=bibtex,
style=numeric]{biblatex}
\usepackage{abstract}
\newcommand{\R}{\mathbb{R}} 
\newcommand{\Z}{\mathbb{Z}} 
 
\DeclareMathOperator{\supp}{supp}        
\DeclareMathOperator{\dist}{dist}        
\DeclareMathOperator{\WF}{WF}            
\DeclareMathOperator{\esssupp}{ess\,supp} 
\DeclareMathOperator{\Vol}{Vol}          
 
\newtheorem{theorem}{Theorem}[section]

\newtheorem{lemma}[theorem]{Lemma}
\newtheorem{proposition}[theorem]{Proposition}

\title{Refinements to the Kuznetsov Formula for Laplace--Beltrami Operators}
\author{Merrick S Chang}
\date{February 2026}
 
\begin{document}
 
\maketitle
 
\begin{abstract}
Wyman and Xi (2023) proved a Kuznetsov-type theorem for the submanifold counting function associated with Laplace--Beltrami eigenvalues. We refine the results of Wyman and Xi, improving their asymptotic: we remove a constant in the codimension-one case, and we prove a new asymptotic in the case that the set of looping times has measure zero.
\end{abstract}
 
\section{Introduction}\label{sec:intro}

The spectral asymptotics of the Laplace--Beltrami operator on a compact
Riemannian manifold $(M,g)$ have been studied since Weyl's foundational work on the distribution of eigenvalues \cite{weyl1911}. Writing \[0 = \lambda_0^2 \le \lambda_1^2 \le \cdots\] for the spectrum of $-\Delta_g$, Weyl's law states that the eigenvalue counting function satisfies \[N(\lambda) = \#\{j : \lambda_j \le \lambda\} \sim (2\pi)^{-n} \omega_n \mathrm{Vol}_g(M)\, \lambda^n,\] where $n = \dim M$ and $\omega_n$ is the volume of the unit ball in $\R^n$. Weyl further conjectured that
$N(\lambda) = c\lambda^n + O(\lambda^{n-1})$;
the size of the remaining lower order terms, and the geometry they encode, has driven much of the subject. Hörmander \cite{hormander1968}, building on Avakumovi\'c \cite{avakumovic1956} and Levitan \cite{levitan1952}, proved this result affirmatively through a careful analysis of the short-time wave propagator. A central insight of Duistermaat and Guillemin \cite{six} (see also Ivrii \cite{seven}) is that this remainder improves to $o(\lambda^{n-1})$ precisely when the periodic geodesics form a measure-zero subset of the cotangent bundle: spectral asymptotics are governed by the dynamics of the geodesic flow. This principle- that a measure-zero hypothesis on a distinguished set of orbits upgrades a remainder estimate- is the template on which the present work is patterned.

A parallel theory concerns not the eigenvalues but the concentration of the
eigenfunctions $e_j$, measured here by their averages over a submanifold. For a
closed embedded submanifold $H \subseteq M$ of dimension $d$, the period
integral $\int_H e_j \, dV_H$ quantifies the oscillation of $e_j$ along $H$.
Such quantities first arose in analytic number theory, where the Kuznetsov sum
formula \cite{kuznetsov1980}, together with the closely related work of
Bruggeman \cite{bruggeman1978}, related weighted sums of Fourier coefficients of
automorphic forms to sums of Kloosterman sums. Zelditch \cite{five} recast these ideas in
the language of microlocal analysis, proving for an arbitrary compact Riemannian
manifold the Kuznetsov sum formula
\[N_H(\lambda) := \sum_{\lambda_j \le \lambda} \left| \int_H e_j \, dV_H
\right|^2 = C_{H,M}\, \lambda^{n-d} + O(\lambda^{n-d-1}),\]
in which the leading growth is governed by the codimension $n - d$. This
one-term asymptotic is, for period integrals, the analogue of Weyl's law for eigenvalues.

Wyman and Xi \cite{one} provide results which unify the estimates  of Duistermaat and Guillemin and Zelditch, establishing a two-term Kuznetsov sum formula which reveals a lower-order oscillatory term, dependent on the geodesic flows. We relay two statements sharpening their work. Write $\dot{N}^*H := N^*H \setminus 0$ for the conormal bundle of $H$ in $M$ with its zero section removed, and let $\Phi_t$ denote the homogeneous geodesic (Hamiltonian) flow on $T^*M \setminus 0$. Following \cite{one}, let \[\mathcal{T}_0 = \left\{ t \in \R ~\middle|~ \Phi_t(x,\xi) \in \dot{N}^*H \textrm{ for some } (x,\xi) \in \dot{N}^*H\right\} \cup \{0\}.\] Set
\[\Sigma_t = \left\{(x,\xi) \in \dot{N}^*H ~\middle|~ \Phi_t(x, \xi) \in \dot{N}^*H \textrm{ and } d\Phi_t \textrm{ is a linear isomorphism}\right\}\]
denote by $S\Sigma_t$ and $SN^*H$ the unit bundles of $\Sigma_t$ and $N^*H$ respectively. Wyman and Xi (2023) show the following:
\begin{proposition}\label{prop:wx-measure0}
If $S\Sigma_t$ is a measure-zero subset of $SN^*H$ for each $t \in \mathcal{T}_0$, and $\mathcal{T}_0$ is countable, then
\[N_H(\lambda) = C_{H,M}\, \lambda^{n-d} + o(\lambda^{n-d-1}) + C.\]
\end{proposition}
\noindent cf. the celebrated Duistermaat--Guillemin theorem \cite{six}, which gives Weyl's law improvements in the case that the set of covectors belonging to closed orbits of the geodesic flow is measure zero; their technique leaves the constant $C$ in the case $n-d = 1$ but they conjecture, based on empirical observation of the case $M = T^2$ and $H = S^1$, that it is superfluous. We ran additional tests in the case of $M$ a sphere and $H$ a family of longitudinal circles, which seemed to offer additional evidence that the constant is removable.
\begin{figure}[H]
    \centering
    \includegraphics[width=0.9\linewidth]{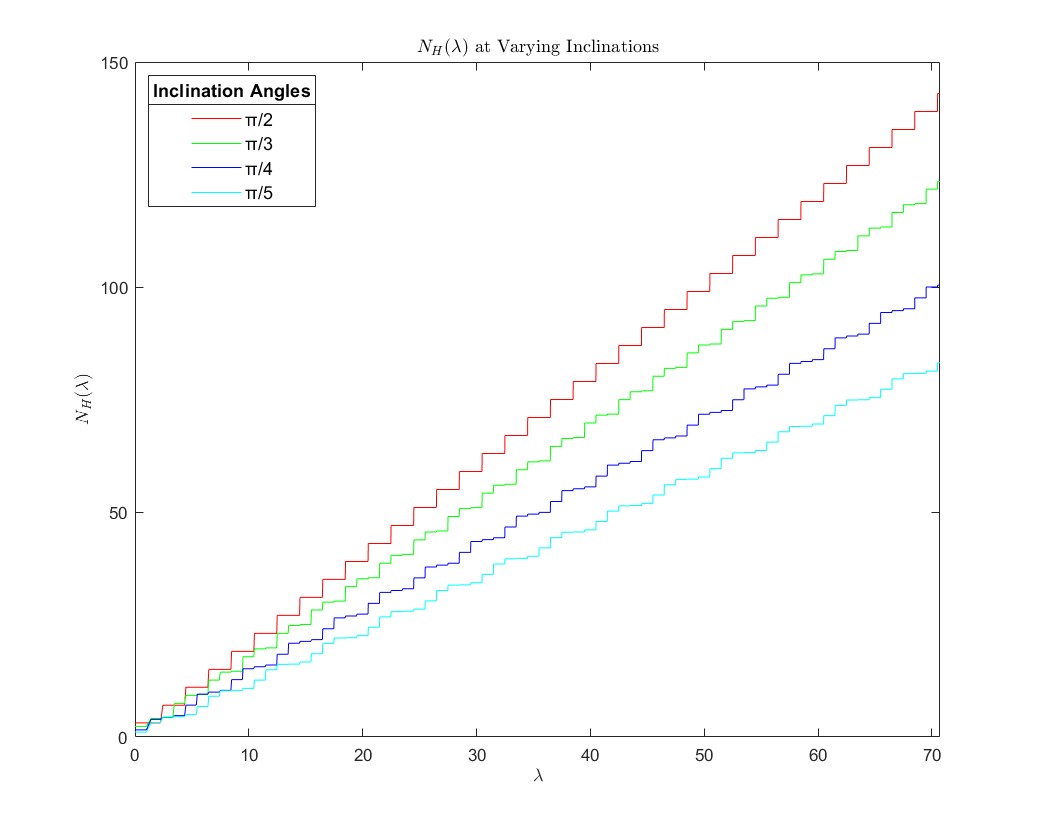}
    \caption{The functions $N_H(\lambda)$ for longitudinal circles $H$ of various inclinations. Visibly, the $y$-intercepts of the linear term seem to be $0$.}
    \label{fig:longitudinal}
\end{figure}
Indeed, our first result is to prove the following in Section~\ref{sec:codim1} through use of the initial result combined with some short-time heat kernel asymptotics:
\begin{theorem}\label{thm:codim1}
Let $(M,g)$ be a compact Riemannian manifold and $H$ an embedded compact submanifold with $\dim H = \dim M - 1$. Assume that
\[N_H(\lambda)=  C_{H,M}\, \lambda + C + o(1).\]
Then
    \[N_H(\lambda)=  C_{H,M}\, \lambda + o(1).\]
\end{theorem}

A relatively minor note is that the compactness assumption on $H$ can likely be essentially removed. As the integral $\left|\int_H e_j dV_H \right|^2$ need not be defined it would need to be replaced with the integral $\left|\int_H \chi e_j dV_H \right|^2$ for some compactly supported function $\chi$ on $H$, but otherwise compactness of $H$ is not a direct requirement.

Second, we substantially weaken the hypothesis of Proposition \ref{prop:wx-measure0} in Section 3:
\begin{theorem}\label{thm:measure0}
Assume that the set of looping directions,
\[\left\{(x, \xi) \in S\dot{N}^*H ~\middle|~ \Phi_t(x, \xi) \in S\dot{N}^*H \textrm{ for some } t>0\right\},\]
is measure zero in $SN^*H$. Then
\[N_H(\lambda) = C_{H,M}\, \lambda^{n-d} + o(\lambda^{n-d-1}).\]
\end{theorem}
\noindent In general the assumption that the set of looping times is measure zero is the weakest possible to get an asymptotic of this type in the literature; in particular, we have eliminated the hypothesis that $\mathcal{T}_0$ be countable. Our proof closely follows the arguments of Wyman (2017), which in turn borrow from Ivrii \cite{seven}. Note that through use of Theorem \ref{thm:codim1} we have eliminated any constant term in the $\mathrm{codim} H = 1$ case.

\section{Proof of Theorem~\ref{thm:codim1}}\label{sec:codim1}
 
\subsection{Setup}
By assumption,
\begin{align*}\int_{-\infty}^\infty e^{-t\lambda^2}\, dN_H(\lambda) &= \int_{-\infty}^\infty 2t\lambda\, e^{-t\lambda^2} N_H(\lambda)\, d\lambda \\&= \int_{0}^\infty 2t\, e^{-t\lambda^2}\left(C_{H,M}\lambda^2 + C \lambda + r(\lambda)\lambda\right) d\lambda
\end{align*}
after integration by parts, where $r(\lambda) = o(1)$. Now
\[\int_0^\infty 2t\, e^{-t\lambda^2}\left(C_{H,M} \lambda^2 + C\lambda\right) d\lambda = 2t\left(\frac{C_{H,M}\sqrt{\pi}}{4t^{3/2}} + \frac{C}{2t}\right) = \frac{C_{H,M}\sqrt{\pi}}{2t^{1/2}} + C.\]
Additionally, for sufficiently large $x$ and arbitrarily small $c$,
\begin{align*}
2\int_{0}^\infty \lambda t\, e^{-t\lambda^2} r(\lambda)\lambda \, d\lambda
&\leq 2\int_0^x \alpha \lambda t\, e^{-t\lambda^2}\, d\lambda + 2\int_x^\infty c \lambda t\, e^{-t\lambda^2}\, d\lambda\\
&= c\, e^{-tx^2}-\alpha e^{-tx^2} + \alpha,
\end{align*}
where $\alpha = \sup_{\lambda \in [0,x]} r(\lambda)$. This expression approaches $c$ as $t\to 0$, but because the choice of $c$ is arbitrary we are done. Thus it suffices to show that
\begin{align*}
\int_{-\infty}^\infty e^{-t\lambda^2}\, dN_H(\lambda) &= \lim_{\lambda \to \infty} \sum_{\lambda_j \leq \lambda} e^{-t\lambda_j^2} \int_H\int_H e_j(x)\overline{e_j(y)}\, dV_H(x)\, dV_H(y)\\
&=\int_H\int_H \lim_{\lambda \to \infty} \sum_{\lambda_j \leq \lambda }  e^{-t\lambda_j^2} e_j(x)\overline{e_j(y)}\, dV_H(x)\, dV_H(y)\\
&=\int_H \int_H e^{t\Delta_g}(x,y)\, dV_H(x)\, dV_H(y)\\&= C_{H,M} t^{-1/2} +\text{vanishing terms}
\end{align*}
as $t \to 0$.
 
\subsection{Local Asymptotics}
 
The following short-time asymptotics are from \cite{three}. Given a Riemannian manifold $M$ of dimension $n$, the heat kernel $p_M(t, x,y)$ uniformly approaches
\[\left( \frac{1}{2\pi t} \right)^{n/2} e^{-\dist(x,y)^2/2t} \sum_{m=0}^{\infty} H_m(x,y)\, t^m\]
save for the cut locus (i.e. the set of points $(x,y)$ such that there is more than one geodesic connecting $x$ and $y$), where each $H_m(x,y)$ is a smooth function with $H_0(x,y) > 0$ and $H_0(x,x) = 1$.
 
Note that in general the heat kernel behaves poorly at the cut locus. In the compact case, however, we have the following additional bound globally:
\[p_M(t,x,y) \leq \frac{C}{t^{(2n-1)/2}}\, e^{-\dist(x,y)^2/2t},\]
so we may comfortably ignore this bad behavior and focus on the neighborhood of the diagonal.
 
We now put the calculation into suitable local form.
 
\begin{lemma}\label{lem:localform}
For some sufficiently small open $U \subseteq \R^n$,
\begin{multline*}\int_U e^{t\Delta_g}(x,y)\, dV_H(y) = \frac{1}{(2\pi t)^{n/2}}\int_U \Bigg(e^{\frac{-|y'|^2 + |f(y')|^2}{2t}} \\  * \sum_{m=0}^{\infty} H_m(0,y)\, t^m \sqrt{1+|f(y')|^2}\bigg)\, dy' \\
+ \textrm{vanishing terms}
\end{multline*}
$f(0) = 0$ with $f$ smooth.
\end{lemma}
 
\begin{proof}
For each point $x \in H$, take $U \subseteq H$ to be a sufficiently small open neighborhood of $x$. By geodesic normal coordinates, we write an embedding, not necessarily isometric, $h: U \to \R^{n}$, $n = \dim M$. By the implicit function theorem, we may write $h = (x_1, x_2, \dots, x_{\dim H}, \varphi(x))$ with $\varphi(0) = 0$ and $\nabla \varphi(0) = 0$. Thus
\begin{align*}(dh)^Tg\,(dh) &= \begin{pmatrix} 1 & 0 & \dots & 0 & \frac{\partial \varphi}{\partial x_1}\\
0 & 1 & \dots & 0 & \frac{\partial \varphi}{\partial x_2}\\
\vdots & \vdots & \ddots &\vdots &  \vdots   \\
0 & 0 & \dots & 1 &  \frac{\partial \varphi}{\partial x_{n-1}} \end{pmatrix} g \begin{pmatrix} 1 & 0 & \dots & 0\\
0 & 1 & \dots & 0 &\\
\vdots & \vdots & \ddots &\vdots  \\
0 & 0 & \dots & 1 \\
\vdots & \vdots & & \vdots\\
\frac{\partial \varphi}{\partial x_1} & \frac{\partial \varphi}{\partial x_2} & \dots & \frac{\partial \varphi} {\partial x_{n-1}}
\end{pmatrix}
\end{align*}
where
\begin{align*}g &= I_{\dim M}+ Y\\
&= I_n + \begin{pmatrix} Y_{1,1}(x) & Y_{2,1}(x) & \dots & Y_{n, 1}(x) \\
Y_{1,2}(x) & Y_{2,2}(x) & \dots & Y_{n,2}(x)\\
\vdots & \vdots & \ddots & \vdots\\
Y_{1, n}(x) & Y_{2, n}(x) & \dots & Y_{n,n}(x)
\end{pmatrix}
\end{align*}
with $Y_{i,j} = O(\|x\|^2)$. Denote
\[S_{i,j} = Y_{i,j} + Y_{n,j}\frac{\partial \varphi}{\partial x_i}.\]
Thus,
\begin{align*}
Y\,(dh) &= \begin{pmatrix} Y_{1,1}(x) & Y_{2,1}(x) & \dots & Y_{n, 1}(x) \\
Y_{1,2}(x) & Y_{2,2}(x) & \dots & Y_{n,2}(x)\\
\vdots & \vdots & \ddots & \vdots\\
Y_{1, n}(x) & Y_{2, n}(x) & \dots & Y_{n,n}(x)
\end{pmatrix}\begin{pmatrix} 1 & 0 & \dots & 0\\
0 & 1 & \dots & 0 &\\
\vdots & \vdots & \ddots &\vdots  \\
0 & 0 & \dots & 1 \\
\frac{\partial \varphi}{\partial x_1} & \frac{\partial \varphi}{\partial x_2} & \dots & \frac{\partial \varphi} {\partial x_{n-1}}
\end{pmatrix}\\
&= \begin{pmatrix}
S_{1,1}  & S_{2,1}  & \dots & S_{n,1}\\
S_{1,2} & S_{2,2} & \dots & S_{n,2}\\
\vdots & \vdots & \ddots & \vdots\\
S_{1,n-1} & S_{2, n-1}& \dots & S_{n,n-1}
\end{pmatrix} = S.
\end{align*}
Because each entry in $(dh^T)S$ vanishes to order $||x||^2$ at the origin, only the identity matrix contributes to the higher order terms and we get \[dV_H = \sqrt{\det g} = \sqrt{1+|f|^2} dx_1 \wedge dx_2 \dots \wedge dx_{n-1}\]
where $f$ vanishes to order at least $x$.
\end{proof}
 
It will suffice to prove the following.
 
\begin{lemma}\label{lem:Cu}
\begin{align*} X &= \frac{1}{(2\pi t)^{n/2}}\int_U e^{\frac{-|y'|^2 -|f(y')|^2}{2t}}  \sum_{m=0}^{\infty} H_m(0,y)\, t^m \sqrt{1+|f(y')|^2}\, dy'\\
&= C_U\, t^{-1/2} + \textrm{vanishing terms}
\end{align*}
where $C_U$ depends only on the ball $U$ and the associated function $f$, and $x \in U$, as $t \to 0$.
\end{lemma}
 
\begin{proof} From the properties of the multivariate normal distribution,
\[(2\pi t)^{-n/2}\int_{|y'| \leq 1}  e^{\frac{-|y'|^2}{2t}} dy' \leq 2 \pi \sqrt{t}.\]
The proof proceeds through a series of estimates of differences. Via the mean value theorem, we claim the estimate
\[\sqrt{1+|f(y')|^2} -1 \leq C_0\,|y'|^2.\]
Thus
\begin{multline*} \int_{|y'|\leq 1} \frac{1}{(2\pi t)^{n/2}} C|y'|^2 e^{\frac{-|y'|^2}{2t}}\, dy'
\\
\geq \frac{1}{(2\pi t)^{n/2}}\bigg( \int_{|y'| \leq 1} e^{\frac{-|y'|^2}{2t}} \sqrt{1+| f(y')|^2}\, dy' \\ - \int_{|y'| \leq 1}  e^{\frac{-|y'|^2}{2t}}\, dy'\bigg).
\end{multline*}
Converting to polar and then substituting $r = \sqrt{2tu}$, so that $dr = \left(\tfrac{2t}{u}\right)^{1/2} du$,
\[ \int_{0}^1 \frac{1}{(2\pi t)^{n/2}} C r^{n+1}\, e^{\frac{-r^2}{2t}}\, dr = \int_0^{1/\sqrt{2t}} (2\pi t)^{1/2} C u^{(n+1)/2} e^{-u}\, du \to 0\]
as $t \to 0$.
 
Now we will prove
\[\int_{|y'| \leq 1}(2\pi t)^{-\frac{n}{2}} e^{-\frac{|y'|^2}{2t}} \left(1 - e^{-\frac{|f(y')|^2}{2t}} \right)\sqrt{1+|f(y')|^2}\,  dy' \to 0.\]
We have that, for some constants $C_1$, $C_2$,
\[ 0 \leq 1-e^{-\frac{|f(y')|^2}{2t}} \leq 1-e^{-\frac{C_1^2|f(y')|^2}{2t}} \leq \min \left\{1,\, C_2\frac{|y'|^4}{2t} \right\}.\]
So we split the integral into the parts
\[(2\pi t)^{-n/2} \int_{U} C\alpha\, e^{-\frac{|y'|^2}{2t}}\, dy'\; \; \; \; \;\textrm{and} \; \; \; \; \; (2\pi t)^{-n/2}\int_V C\alpha \frac{|y'|^4}{2t}\, e^{-\frac{|y'|^2}{2t}}\, dy',\]
where $\alpha = \max_{|y'|\leq 1} \sqrt{1+|f(y')|^2}$.
 
Let $C_2$ be such that
\[
	|f(y')| \leq C_2|y'|^2 \qquad \text{ when } |y'| \leq 1,
\]
hence
\[
	\frac{|f(y')|^2}{2t} \leq \frac{C_2^2 |y'|^4}{4t} \qquad \text{ when } |y'| \leq 1.
\]
Cut the integral into two parts,
\begin{align*}
	A &= (2\pi t)^{-\frac n 2} \int_{|y'| \leq ct^{1/4}} e^{-|y'|^2/2t}\, \big(1 - e^{-|f(y')|^2/2t}\big) \, dy', \\
	B &= (2\pi t)^{-\frac n 2} \int_{|y'| > ct^{1/4}} e^{-|y'|^2/2t}\, \big(1 - e^{-|f(y')|^2/2t}\big) \, dy',
\end{align*}
where $c$ is such that $|y'| \leq c t^{1/4}$ implies $|f(y')|^2/4t \leq 1$. Then
\begin{align*}
	B &= (2\pi t)^{-\frac n 2} \int_{1 \geq |y'| > ct^{1/4}} e^{-|y'|^2/2t}\, \big(1 - e^{-|f(y')|^2/2t}\big) \, dy' \\
	&\leq (2\pi t)^{-\frac n 2} \int_{|y'| > ct^{1/4}} e^{-|y'|^2/2t} \, dy' \\
	&\lesssim (2\pi t)^{-\frac n 2} \int_{ct^{1/4}}^\infty r^{n-2}\, e^{-r^2/2t} \, dr.
\end{align*}
 
Let $r = \sqrt{2tu}$, so that $dr = \left(\tfrac{2t}{u}\right)^{1/2} du$, and so
\begin{align*}
	B &\lesssim t^{-\frac n 2} \int_{c/\sqrt{t}}^\infty (2tu)^{n/2-1} e^{-u} \left(\tfrac{2t}{u}\right)^{1/2}  \, du \\
	&\lesssim t^{-\frac {1}{2}} \int_{c/\sqrt{t}}^\infty u^{n-2}\, e^{-u} \, du.
\end{align*}
 
Thus we focus on the integral $A$. Now, changing variables, again let  $r = \sqrt{2tu}$, so that $dr = \left(\tfrac{2t}{u}\right)^{1/2} du$,
\begin{align*}
 (2\pi t)^{-n/2}\int_V \frac{|y'|^4}{2t}\, e^{-\frac{|y'|^2}{2t}}\, dy' &\leq  (2\pi t)^{-n/2}\int_0^{1} e^{-\frac{r^2}{2t}} r^{n+2}\, dr\\
 &\leq (2\pi t)^{-n/2 - 1}\int_0^{1/\sqrt{2t}} \pi  (2tu)^{n/2 + 3/2} u^{-1/2}\, e^{-u^2}\, du,
\end{align*}
which vanishes.
 
Recalling that $H_0(0,0) = 1$ so $H_0(0,y') - 1 = 0$, finally
\begin{multline*}
(2\pi t)^{-n/2}\left|\int_{|y'| \leq 1} e^{\frac{|f(y')|^2 + |y'|^2}{2t}} \left(\sum_{m = 0}^\infty H_m(0,y')\, t^m  - 1 \right) \sqrt{1+|f|^2}\, dy' \right|\\
\leq (2\pi t)^{-n/2} \Bigg(\sum_{m = 1}^\infty  t^{m} \int_{|y'|\leq 1} M_m\, e^{\frac{|f(y')|^2 + |y'|^2}{2t}}\, \sqrt{1+|f|^2}dy'\\
+ M_0 \int_{|y'| \leq 1} e^{\frac{|f(y')|^2+|y'|^2}{2t}} \sqrt{1+|f|^2}dy'\Bigg)
\end{multline*}
where $M_m = \sup_{|y'| \leq 1} \big(H_m(0,y')\sqrt{1+|f(y')|^2}\big)$ for $m > 0$ and some $M_0$. Thus the required asymptotic holds for appropriate constant.
\end{proof}
\section{Proof of Theorem~\ref{thm:measure0}}\label{sec:measure0}

We now proceed with the proof of Theorem~\ref{thm:measure0}. Consider the following Tauberian theorem: 
\begin{proposition}[Levitin- in Appendices \cite{four}]\label{prop:levitin}
Let $N$ be a real-valued, monotone non-decreasing function on $\R$ such that $N(\lambda) = 0$ for $\lambda \leq 0$ and $N(\lambda) = o(\lambda^p)$. Fix $\gamma$ real-valued on $\R$ such that
\begin{enumerate}
    \item $\gamma \in \mathcal{S}(\R)$;
    \item $\gamma(\lambda) > 0$;
    \item $\hat{\gamma}(0) = \int \gamma(\lambda)\, d\lambda = 1$;
    \item $\supp \hat{\gamma}$ is compact;
    \item and $\gamma$ is even.
\end{enumerate}
Suppose that
$(N'*\gamma)(\lambda) = O(\lambda^\nu)$ and
\[(N'*\rho)(\lambda) = o(\lambda^\nu)\]
for every $\rho$ with $\hat{\rho} \in C_c^\infty(\R)$ and $\supp \hat{\rho} \subset (0, \infty)$. Then
\[N(\lambda) = (N * \gamma)(\lambda) + o(\lambda^\nu).\]
\end{proposition}
 
Now Wyman and Xi \cite{one} prove the following.
\begin{proposition}\label{prop:wx-main}
Let $\gamma$ be a real-valued Schwartz function on $\R$ with small Fourier support, with $\hat \gamma$ real-valued and $\hat \gamma(0) = 1$. Then
\[(N' * \gamma)(\lambda) = (2\pi)^{-n+d}\, \Vol(H)\, \Vol(S^{n-d-1})\, \lambda^{n-d-1} + O(\lambda^{n-d-3})\]
and thus
\[N * \gamma(\lambda) = C_{H,M} \lambda^{n-d} + O(\lambda^{n-d-2} \log \lambda) + C.\]
\end{proposition}
So we need only prove the second condition and then invoke Theorem \ref{thm:codim1} to prove the result. Specifically, we will show the following.
\begin{lemma}\label{lem:second-cond}
Under the assumptions of Theorem 1.3, the inequality
\[\sum_{j} \rho(\lambda - \lambda_j) \left|\int_H e_j \, dV_H \right|^2 \leq c\,\lambda^{n-d-1} + C\,\lambda^{n-d-2}\]
holds, where $C$ depends only on $H$ and $c$, $c$ is arbitrary, and $\supp \hat{\rho} \subset (0, \infty)$ is compact.
\end{lemma}
This shows the second condition, that $(N'*\rho)(\lambda) = o(\lambda^\nu)$, $\nu = n-d-1$.

Assume without loss of generality that $\hat{\rho} \subseteq [\delta, T]$ for some appropriate $\delta$ and $T$. We will make use of the following propositions in the proof. First, let
\begin{align*}\mathcal{L}_H([\delta, T]) = \{(x, \xi) \in SN^*H : {}&\Phi_t(x, \xi) = (y, \eta) \in SN^*H \\& \textrm{ for some } t\in [\delta,T]\}.
\end{align*}
\begin{proposition}[\cite{two}]\label{prop:bB}
Fix $T > 1$ and $\epsilon > 0$. There exist $b, B \in \Psi_{\textrm{cl}}^0(M)$, where $\Psi_{cl}^0(M)$ is the space of classical pseudodifferential operators homogeneous of degree 0 on $M$, and supported in a neighborhood of $\supp V_H$, with the following properties.
\begin{enumerate}
    \item $b(x, D) + B(x,D) = I$;
    \item in suitable local coordinates,
    \[\int_{\R^d} \int_{S^{n-d-1}} |b(x', \omega)|^2\, d\omega\, dx' < \epsilon,\]
    where $b(x, \xi)$ denotes the principal symbol of $b(x,D)$;
    \item the essential support (henceforth denoted $\esssupp$) of $B(x,D)$ contains no elements of $\mathcal{L}_H([\delta, T])$.
\end{enumerate}
\end{proposition}
 
Second, for a distribution $u$, let $\WF(u)$ denote its wavefront set:
\begin{proposition}[\cite{two}]\label{prop:wf-smooth}
Let $u$ and $v$ be distributions on $M$ with
\[\WF(u) \; \cap \; \WF(v) = \emptyset.\]
Then
\[t \mapsto \int_M e^{it\sqrt{-\Delta_g}} u(x)\, \overline{v(x)}\, dx\]
is a smooth function of $t$.
\end{proposition}
 
And third:
\begin{proposition}[\cite{two}]\label{prop:wyman-symbol}
Let $b(x, \xi)$ be a smooth function for $\xi \neq 0$, homogeneous of degree $0$ in $\xi$. Define $b(x,D) \in \Psi_{\mathrm{cl}}^0(M)$ by
\[b(x,D)f(x) = \frac{1}{(2\pi)^n} \int_{\R^n} \int_{\R^n} e^{i\langle x-y, \xi \rangle} b(x, \xi)\, f(y)\, dy\, d\xi\]
for $x, y, \xi$ expressed locally. Then
\begin{align*}\sum_{\lambda_j \in [\lambda, \lambda+1]} \left| \int_H b\, e_j \, dV_H\right|^2
&\leq C \left(\int_{\R^d} \int_{S^{n-d-1}} |b(x', \omega)|^2 h(x')^2\, d\omega\, dx'\right) \lambda^{n-d-1} \\ &\; \; \; \; \;\ \;+ C_b\, \lambda^{n-d-2},
\end{align*}
where $C$ is independent of $b$ and $\lambda$, and $C_b$ is independent of $\lambda$ but depends on $b$.
\end{proposition}
 
\begin{proof}[Proof of Lemma~\ref{lem:second-cond}] 
Now,
\begin{align*}\sum_j   \int_H \int_H & \rho(\lambda_j - \lambda)\, e_j(x)\overline{e_j(y)}\, dV_H(x)\, dV_H(y)\\
&= \frac{1}{2\pi} \sum_j \int_{-\infty}^\infty \int_H \int_H \hat\rho(t)\, e^{it(\lambda_j - \lambda)} e_j(x)\overline{e_j(y)}\, dV_H(x)\, dV_H(y)\, dt\\
&=\frac{1}{2\pi} \sum_j \int_{-\infty}^\infty \int_H \int_H \hat\rho(t)\, e^{-it\lambda}\, e^{it\sqrt{-\Delta_g}} e_j(x)\overline{e_j(y)}\, dV_H(x)\, dV_H(y)\, dt.
\end{align*}
In view of Proposition~\ref{prop:bB}, denote the operator
\[X_{\lambda}(x,y) = \frac{1}{2\pi} \int_{-\infty}^\infty \hat{\rho}(t)\, e^{-it \lambda}\, e^{it \sqrt{-\Delta_g}}(x,y)\, dt.\]
We have
\begin{align*}
\int_H \int_H X_{\lambda}(x,y)\, dV_H(x)\, dV_H(y) = &\int_H\int_H B X_{\lambda}(x,y) B^*\, dV_H(x)\, dV_H(y) \\
&+\int_H \int_H B X_{\lambda}(x,y) b^*\, dV_H(x)\, dV_H(y)\\
&+ \int_H \int_H b X_{\lambda}(x,y) B^*\, dV_H(x)\, dV_H(y)\\
&+\int_H \int_H b X_{\lambda}(x,y) b^*\, dV_H(x)\, dV_H(y).
\end{align*}
 
We claim the first three terms are $O(\lambda^{-N})$ for arbitrary $N$. We prove the bound for the first term; the second and third are handled identically, since $X_\lambda$ is self-adjoint. Interpreting $V_H$ as a distribution on $M$, we write formally
\begin{align*}
    \int_H \int_H B X_\lambda B^*(x,y)\, dV_H(y)\, dV_H(x) &= \int_M \int_M  X_{\lambda}(x,y)\, B^*V_H(y)\, B^*V_H(x)\, dx\, dy\\
    &= \frac{1}{2\pi} \int_{-\infty}^\infty \bigg(\hat{\rho}(t)\,e^{-it\lambda}\\& \ \; \; \; \; \;\; \; \;\ \; \;* \int_M e^{it\sqrt{-\Delta_g}} (B^* V_H)(x)\, \overline{B^*V_H(x)}\, dx \bigg)\, dt.
\end{align*}
Thus by Proposition~\ref{prop:wf-smooth}, it suffices that
\[\WF\!\left(e^{it \sqrt{-\Delta_g}} B^*V_H\right) \cap \WF(B^* V_H) = \emptyset,\]
and the result follows via integration by parts. Now
\[\WF(B^* V_H) \subset \esssupp B \cap N^* H.\]
Suppose that $(x, \xi)$ is a unit covector in $\WF(B^* V_H)$. Then $\Phi_t(x,\xi) \notin SN^* H$ for any $\delta \leq t \leq T$. By propagation of singularities,
\[\WF\!\left(e^{it \sqrt{-\Delta_g}} B^* V_H\right) = \Phi_t\, \WF(B^* V_H),\]
and thus
\[\WF\!\left(e^{it \sqrt{-\Delta_g}} B^* V_H\right) \cap \WF(B^* V_H) = \emptyset \qquad \delta \leq t \leq T.\]
 
So consider
\[\int_H \int_H b X_{\lambda}(x,y) b^*\, dV_H(x)\, dV_H(y) = \sum_j \rho(\lambda - \lambda_j) \left|\int_H b(x,D) e_j(x)\, dV_H(x) \right|^2.\]
which is the only term we have not covered. 

Note that $\rho$ satisfies the bounds
\[|\rho(z)| \leq C_N(1+|z|)^{-N} \qquad \textrm{for } N = 1, 2, \dots.\]
We dominate $|\rho|$ by a step function $\sum_{k \in \Z} a_k\, \mathbf{1}_{[k, k+1]}$ satisfying $a_k = \sup_{[k, k+1]} |\rho|$. 
Applying Proposition~\ref{prop:wyman-symbol},
\begin{align*}
\sum_j \rho(\lambda - \lambda_j) \left|\int_H b(x,D) e_j\, dV_H \right|^2 &\leq C \sum_{k \in \Z} a_k \bigg(\epsilon\, (|\lambda + k| + 1)^{n-d-1}
\\&\;\;\;\;\;\;\;\;+ C_b(|\lambda + k| +1)^{n-d-2}\bigg)
\\&\leq \epsilon C\,\lambda^{n-d-1} + C_b\,\lambda^{n-d-2},
\end{align*}
which concludes the proof.
\end{proof}

\section{AI Acknowledgments}
Claude was used to generate some of the introduction, clean up the writing, and check MatLab code. All substantive results are the work of the author alone.

\section{Acknowledgments}
I'd like to thank my advisor Emmett Wyman for his insights and helping spot errors in the proof.
 
\printbibliography

\end{document}